\documentclass[12pt]{amsart}
\usepackage{amsmath}
\usepackage{amsthm}
\usepackage{amssymb}
\usepackage{amscd}
\usepackage{bbold}
\usepackage[pdftex]{graphicx}
\usepackage{enumerate }
\usepackage{subfigure}

\newtheorem{lem}{Lemma}[section]
\newtheorem{thm}[lem]{Theorem}
\newtheorem{prop}[lem]{Proposition}

\theoremstyle{definition}
\newtheorem{defn}[lem]{Definition}
\newtheorem{example}[lem]{Example}

\newtheorem{rem}[lem]{Remark}

\def\benm{\begin{enumerate}}            % Begin enumerate command
\def\eenm{\end{enumerate}}              % End enumerate command
\title{ B-Frames, B-Riesz bases, and Their tensor products }

\author{Abdelkrim Bourouihiya$^2$}

\address{$^2$ Department of Mathematics, Nova Southeastern University
3301 College Avenue, Fort Lauderdale, Florida, USA}
\email{ab1221@nova.edu}

\author{Chaimae Mezzat$^1$ $^*$}

\address{$^1$ Department of Mathematics, Faculty of Science, Ibn Tofail University,
B.P. 133, Kenitra, Morocco.}
\email{chaimae.mezzat@uit.ac.ma}

\date{\today}


\begin{document}
\newcounter{bean}
\maketitle
 \begin{abstract}
Like g-frames, b-frames were introduced to generalize the concept of frames, allowing for broader applications in signal processing and other fields. The advantage of b-frames resides in their simpler definition, which may lead to reduced processing times. In this paper, we define dual b-frames and b-Riesz bases—which were not precisely defined in previous literature—and provide several characterizations of b-Riesz bases. We prove that the tensor product of two sequences, each lying in a Hilbert space, constitutes a b-frame (or a b-Riesz basis) if and only if both components of the product are b-frames (or b-Riesz bases). Finally, we establish a correspondence between b-frames and g-frames and propose a process for constructing frames induced by b-frames.
\end{abstract}

%%%%%%%%%%%%%%%%%%%%%%%%%%%%%%%%%
\section{Introduction}
\label{sec:1}
Throughout this paper, \(\mathcal{B}\) denotes a Banach space, while \(\mathcal{H}\) and \(\mathcal{Z}\) represent two separable Hilbert spaces. If $X$ and $Y$ are two Banach spaces, $\mathcal{L}(X,Y)$ denotes the space of bounded operators $: X \to Y$. We denote by $\mathcal{K}$ a set of countable indices.

A sequence  $\{f_k\}_{k \in \mathcal{K}} \subset \mathcal{H}$ is a frame with the frame bounds $A,B>0$ if
\begin{eqnarray}\label{frame}
A \|h \|^2_{\mathcal{H}} \leq \sum_{k \in \mathcal{K}} | \langle h , f_k \rangle_{\mathcal{H}} |^2_{\mathcal{H}}  \leq  B \|h\|^2 _{\mathcal{H}}, \quad \forall h \in \mathcal{H}.
\end{eqnarray}

A g-frame for  $\mathcal{Z}$ with respect to a sequence of Hilbert spaces $ \{\mathcal{Z}_k\}_{k \in \mathcal{K}} $ is a family of bounded linear operators $ \{ \Lambda_k \in \mathcal{L}(\mathcal{Z}, \mathcal{Z}_k )\}_{k \in \mathcal{K}} $ such that:
\begin{eqnarray}\label{gframe}
A\|z\|^2_\mathcal{Z} \leq  \sum_{k \in \mathcal{K}}\| \Lambda_k(z) \|^2_{\mathcal{Z}_k} \leq B \| z \|^2_\mathcal{Z}, \quad \forall z \in \mathcal{Z}.
\end{eqnarray}
$A$ and $B$ are called the g-frame bounds of the g-frame $ \{ \Lambda_k \}_{k \in \mathcal{K}} $.

Let  $\mathfrak{b}: \mathcal{H }\times \mathcal{B} \to \mathcal{Z }$ be a continuous bilinear mapping, i.e, there is $M>0$ such that
\begin{eqnarray*}
\|\mathfrak{b}(h,x)\|_{\mathcal{Z}} \leq M \|h \|_{\mathcal{H}} \|x \|_{\mathcal{B}}, \quad \forall (h,x) \in \mathcal{H }\times \mathcal{B}
\end{eqnarray*}

Fixing $x \in \mathcal{B}$ and $z \in \mathcal{Z}$, $h  \rightarrow \langle \mathfrak{b}(h,x), z \rangle_{\mathcal{Z}} $ defines a bounded linear form on $\mathcal{H}$. Hence, Riesz representation theorem implies the existence of a unique vector $v \in  \mathcal{H}$ such that
\begin{eqnarray*}
\langle \mathfrak{b}(h,x), z \rangle_{\mathcal{Z}} = \langle h, v \rangle_{\mathcal{H}}, \forall h \in \mathcal{H}.
\end{eqnarray*}
We define $ \langle z / x \rangle= v$ the  so called b-frame product of $z$ and $x$. Hence, we can write
\begin{eqnarray}\label{bproduct}
\langle \mathfrak{b}(h,x), z \rangle_{\mathcal{Z}} = \langle h, \langle z / x \rangle\rangle_{\mathcal{H}}, \forall (h,x,z) \in \mathcal{H} \times  \mathcal{B} \times   \mathcal{Z}.
\end{eqnarray}

$\{x_k\}_{k>0} \subset \mathcal{B}$ is called a b-frame for $\mathcal{Z}$ if there are two positive constants $A$ and $B$, called the b-frame bounds, such that
\begin{eqnarray}\label{bframe}
A \|z \|^2_{\mathcal{Z}} \leq \sum_{k \in \mathcal{K}} \| \langle z / x_k \rangle \|^2_{\mathcal{H}}  \leq  B \|z\|^2 _{\mathcal{Z}}, \quad \forall z \in \mathcal{Z}.
\end{eqnarray}

In addition to generalizing frames, fusion frames, and other types of frames, g-frames serve as a more robust alternative to traditional frames in some  applications \cite{Kho}. They offer more freedom when using multiresolution analysis \cite{Ism}, as well as a more advanced mathematical language for mapping quantum measurements in quantum detection \cite{Zha}. As another generalization of frames, b-frames were introduced by Ismailov et al. \cite{Ism}. Following this paper, several other publications explored b-frames, their properties, and their extensions to more general classes, including k-b-frames \cite{Mez}. While a g-frame is defined by a sequence of operators paired with a sequence of Hilbert spaces, a b-frame is defined using a single bilinear mapping and three spaces. This simpler definition may translate to fewer required computational steps, lowering processing times and rendering large-scale signal transformations more viable for practical, real-time scenarios.

In Section 2, we define the canonical dual of a b-frame and use it to formulate the b-frame expansion. We also establish a correspondence between a b-frame and a frame. In Section 3, we define b-Riesz bases and provide several characterizations of them. These concepts—the canonical dual, the b-frame expansion, and b-Riesz bases—were not addressed in previous papers on b-frames. In Section 4, we prove that the tensor product of two sequences, each lying in a Hilbert space, constitutes a b-frame (or a b-Riesz basis) if and only if each component sequence is a b-frame (or a b-Riesz basis). In Section 5, we show that every g-frame (or g-Riesz basis) corresponds to a b-frame (or b-Riesz basis), and vice versa. Furthermore, the b-frame and g-frame share the same bounds, and the b-frame operator coincides with the g-frame operator. In Section 6, we prove a theorem that can serve to construct frames for a Hilbert space $\mathcal{Z}$ induced by b-frames for $\mathcal{Z}$ and frames for another Hilbert space $\mathcal{H}$.
%%%%%%%%%%%%%%%%%%%%%%%%%%%%%%%%%%%%%%%%%%%%%%%%%%%%%%%%
\section{Dual b-frame and b-frame expansion }
\label{sec:2}
The b-frame operator $S:\mathcal{Z} \to \mathcal{Z} $ of a b-frame $(\{x_k\}_{k \in \mathcal{K}},\mathfrak{b})$ is defined by
\begin{eqnarray}\label{operator}
S(z)=\sum_{k \in \mathcal{K}} \mathfrak{b}(\langle z / x_k \rangle,x_k ).
\end{eqnarray}
$S$ is a positive, self-adjoint, and bounded invertible operator \cite{Ism}.

We define $\tilde{\mathfrak{b}}=S^{-1}\mathfrak{b}$. Therefore, $\tilde{\mathfrak{b}}: \mathcal{H }\times \mathcal{B} \to \mathcal{Z }$ is a continuous bilinear mapping since
\begin{eqnarray*}
 \|\tilde{\mathfrak{b}}(h,x) \|_{\mathcal{Z}} =   \| S^{-1}\mathfrak{b}(h,x) \|_{\mathcal{Z}}\leq M  \|S^{-1}  \|  \|h \|_{\mathcal{H}} \|x \|_{\mathcal{B}}, \quad \forall (h,x) \in \mathcal{H }\times \mathcal{B}
\end{eqnarray*}

\begin{prop}\label{bdual} Let $\{e_j\}_{j \in \mathcal{K}}$ be an orthonormal basis  for $\mathcal{H}$ and let $\{x_k\}_{k \in \mathcal{K}} \subset \mathcal{B}$.
\begin{enumerate}
\item $(\{x_k\}_{k \in \mathcal{K}},\mathfrak{b})$ is a b-frame for $\mathcal{Z}$ with the b-frame bounds $A$ and $B$  if and only if $\{\mathfrak{b}(e_j,x_k)\}_{j,k \in \mathcal{K}}$ is a frame for $\mathcal{Z}$ with the frame bounds $A$ and $B$.
  \item Suppose $(\{x_k\}_{k \in \mathcal{K}},\mathfrak{b})$ is a b-frame for $\mathcal{Z}$.
  \begin{enumerate}
    \item $S$ is the b-frame operator of $(\{x_k\}_{k \in \mathcal{K}},\mathfrak{b})$, if and only if $S$ is the frame operator of $\{b(e_j,x_k)\}_{j,k \in \mathcal{K}}$.
     \item $(\{x_k\}_{k \in \mathcal{K}},\tilde{\mathfrak{b}})$ is a b-frame for $\mathcal{Z}$ with the frame bounds $B^{-1}$ and $A^{-1}$. We call $(\{x_k\}_{k \in \mathcal{K}},\tilde{\mathfrak{b}})$ the canonical dual b-frame of the b-frame $(\{x_k\}_{k \in \mathcal{K}},\mathfrak{b})$.
    \item For each $z \in \mathcal{Z}$, we have the following b-frame expansion
    \begin{eqnarray}\label{expansion}
z=\sum_{k \in \mathcal{K}}\tilde{\mathfrak{ b }}(\langle z / x_k \rangle,x_k )=\sum_{k \in \mathcal{K}} \mathfrak{ b }(\langle z / x_k \rangle\tilde{ \ },x_k ),
\end{eqnarray}
where $\langle . / . \rangle\tilde{\ }$ is the b-frame product with respect to $\tilde{\mathfrak{ b }}$.
\end{enumerate}
\end{enumerate}
 \end{prop}

\begin{proof} Let $\{e_j\}_{j \in \mathcal{K}}$ be an orthonormal basis  for $\mathcal{H}$ and let $\{x_k\}_{k \in \mathcal{K}} \subset \mathcal{B}$.
\begin{enumerate}
  \item  We have
  \begin{eqnarray*}
 \sum_{k \in \mathcal{K}} \| \langle z/x_k \rangle \|_\mathcal{H}^2 &=&  \sum_{k \in \mathcal{K}}  \sum_{j \in \mathcal{K}} | \langle \langle z/ x_k \rangle, e_j \rangle_\mathcal{H} |^2\\
 & =&  \sum_{k \in \mathcal{K}} \sum_{j \in \mathcal{K}} |  \langle z, \mathfrak{b}(e_j,x_k) \rangle_\mathcal{Z} |^2.
  \end{eqnarray*}
 Therefore, \ref{bframe} holds for $\{x_k\}_{k \in \mathcal{K}}$ if and only if \ref{frame} holds for $\{b(e_j,x_k)\}_{j,k \in \mathcal{K}}$. We conclude that $(\{x_k\}_{k \in \mathcal{K}}, \mathfrak{ b })$ is a b-frame  fro $\mathcal{Z}$  if and only if $\{b(e_j,x_k)\}_{j,k \in \mathcal{K}}$ is a frame for $\mathcal{Z}$ with the same bounds.
 \item Suppose $(\{x_k\}_{k \in \mathcal{K}},\mathfrak{b})$ is a b-frame for $\mathcal{Z}$ and let $S$ be its b-frame operator. By Statement (1), $\{\mathfrak{b}(e_j,x_k)\}_{j,k \in \mathcal{K}}$ is a frame for $\mathcal{Z}$. Let $S'$ be the frame operator of  $\{\mathfrak{b}(e_j,x_k)\}_{j,k \in \mathcal{K}}$.

\begin{enumerate}
  \item  Using \ref{operator} and the continuity of $\mathfrak{b}$, we have
   \begin{eqnarray*}
S(z)=\sum_{k \in \mathcal{K}} \mathfrak{b}(\langle z / x_k \rangle,x_k )&=& \sum_{k \in \mathcal{K}} \mathfrak{b}(\sum_{j \in \mathcal{K}}\langle \langle z / x_k \rangle, e_j \rangle_\mathcal{H} e_j,x_k )\\
&=& \sum_{k \in \mathcal{K}} \sum_{j \in \mathcal{K}} \langle \langle z / x_k \rangle, e_j \rangle_\mathcal{H} \mathfrak{b}(e_j,x_k )\\
&=& \sum_{k \in \mathcal{K}} \sum_{j \in \mathcal{K}} \langle  z , \mathfrak{b}(e_j, x_k)\rangle_\mathcal{Z} \mathfrak{b}(e_j,x_k )= S'(z).
\end{eqnarray*}
Thus, Statement (2.a) is proven.
\item Using \ref{bproduct}, for each $(h,x,z) \in \mathcal{H} \times \mathcal{B} \times \mathcal{Z}$, we have
\begin{eqnarray*}
\langle h, \langle z / x \rangle \tilde{ \ } \rangle_{\mathcal{H}} = \langle \tilde{\mathfrak{b}}(h,x), z \rangle_{\mathcal{Z}} =  \langle S^{-1}\mathfrak{b}(h,x), z \rangle_{\mathcal{Z}},
\end{eqnarray*}
and since $S^{-1}$ is self adjoint,  we  have
\begin{eqnarray*}
\langle h, \langle z / x \rangle \tilde{ \ } \rangle_{\mathcal{H}} =  \langle \mathfrak{b}(h,x), S^{-1}z \rangle_{\mathcal{Z}}= \langle h, \langle S^{-1}z / x \rangle  \rangle_{\mathcal{H}},
\end{eqnarray*}
and so $\langle z / x \rangle \tilde{ \ } = \langle S^{-1}z / x \rangle$. Therefore,
\begin{eqnarray*}
 \sum_{k \in \mathcal{K}} \| \langle z/x_k \rangle\tilde{ \ }\|^2_\mathcal{H} &=& \sum_{k \in \mathcal{K}} \| \langle S^{-1}z / x \rangle\|^2_\mathcal{H}\\
 &=& \sum_{k \in \mathcal{K}} \sum_{j \in \mathcal{K}} |  \langle S^{-1}z, \mathfrak{b}(e_j,x_k) \rangle_\mathcal{Z} |^2\\
  &=& \sum_{k \in \mathcal{K}} \sum_{j \in \mathcal{K}} |  \langle z, S^{-1}\mathfrak{b}(e_j,x_k) \rangle_\mathcal{Z} |^2
  \end{eqnarray*}
 As the canonical dual frame of $\{b(e_j,x_k)\}_{j,k \in \mathcal{K}}$, $\{S^{-1}b(e_j,x_k)\}_{j,k \in \mathcal{K}}$ is a frame with the frame bounds $B^{-1}$ and $A^{-1}$ \cite{Chr}. Using Statement \emph{(1)}, $(\{x_k\}_{k \in \mathcal{K}},\tilde{\mathfrak{b}})$ is then a b-frame with the b-frame bounds $B^{-1}$ and $A^{-1}$.

\item  Using \ref{operator}, for each $z \in \mathcal{Z}$, we have
\begin{eqnarray*}
z&=&\sum_{k \in \mathcal{K}} S^{-1} \mathfrak{b}(\langle z / x_k \rangle,x_k )=\sum_{k \in \mathcal{K}}\tilde{\mathfrak{ b }}(\langle z / x_k \rangle ,x_k )
\end{eqnarray*}
and
\begin{eqnarray*}
z=\sum_{k \in \mathcal{K}} \mathfrak{b}(\langle S^{-1}z / x_k \rangle,x_k )=\sum_{k \in \mathcal{K}}\mathfrak{ b }(\langle z / x_k \rangle \widetilde{ \ },x_k ).
\end{eqnarray*}

\end{enumerate}

\end{enumerate}
\end{proof}

For the remaining of this paper, we define $\mathfrak{b}_0: (h,T) \in \mathcal{H} \times \mathcal{L}(\mathcal{H},\mathcal{Z}) \to \mathfrak{b}_0(h,T)= T(h)$. It is obvious that $\mathfrak{b}_0:\mathcal{H} \times \mathcal{L}(\mathcal{H},\mathcal{Z}) \rightarrow \mathcal{Z}$ is a continuous bilinear mapping.
For each $(h,z,T) \in \mathcal{H} \times \mathcal{Z} \times \mathcal{L}(\mathcal{H},\mathcal{Z})$, we have
\begin{eqnarray*}
\langle T(h), z\rangle_{\mathcal{Z} } =\langle h, T^*(z)\rangle_{\mathcal{H} },
\end{eqnarray*}
and so $\langle z/ T\rangle_{\mathfrak{b_0} } =T^*(z)$. Therefore,  $(\{T_k\}_{k \in \mathcal{K}},\mathfrak{b}_0)$ is a b-frame for $\mathcal{Z}$ if there are two positive constants $A$ and $B$ such that
 \begin{eqnarray}\label{bgframe1}
A||z||^2_{\mathcal{Z}} \leq \sum_{k \in \mathcal{K}} ||  T_k^*(z)  ||^2_{\mathcal{H}}  \leq  B||z||^2 _{\mathcal{Z}}
\end{eqnarray}

The b-frame operator of $(\{T_k\}_{k \in \mathcal{K}},\mathfrak{b}_0)$ is then
\begin{eqnarray*}
S= \sum_{k\in \mathcal{K}} T_k T_k^*
\end{eqnarray*}

\begin{prop}\label{Gramian} [Gramian b-frame] Let $ U:\mathcal{H} \to \mathcal{Z}$ be a bounded surjective operator.  Let $ V:\mathcal{Z} \to \mathcal{Z}$ be a  bounded operator such that $||V|| < 1$.
\begin{enumerate}
  \item $(\{x_k= V^k U\}_{k \geq 0}, \mathfrak{b}_0)$ is a b-frame with the  b-frame bounds
  \begin{eqnarray*}
  A= \inf \{ \| U^*(z) \|^2_\mathcal{H}:  \|z \|_\mathcal{Z} =1 \} \mbox{ and } B= \frac{\|U \|^2}{1-\|V\|^2}
  \end{eqnarray*}
  and the b-frame operator
  \begin{eqnarray}\label{GramianO}
S= \sum_{k \geq 0}  V^k U U^* (V^*)^k
\end{eqnarray}
\item If $U U^*=I_{\mathcal{Z}}$, the identity operator of $\mathcal{Z}$, and $V$ is a normal operator, then the system $(\{x_k= V^k U\}_{k \geq 0}, \mathfrak{b}_0)$ is a b-frame  with the b-frame operator $S=(I_{\mathcal{Z}} -VV^*)^{-1}$ and  b-frame bounds
    \begin{eqnarray*}
A = \frac{1}{\|I_{\mathcal{Z}} -VV^* \| }  \mbox{ and } B=||( I_\mathcal{Z} -VV^* )^{-1}||.
\end{eqnarray*}
  \end{enumerate}

The operator $S$, defined by \ref{GramianO}, is an example of a Gramian operator associated with a discrete-time linear system \cite{Chr}. That is why we call  $(\{x_k= V^k U\}_{k \geq 0}, \mathfrak{b}_0)$ a Gramian b-frame.
\end{prop}

\begin{proof}\emph{(1)} The series in \ref{GramianO} is normally convergent, and so $S$ is a bounded operator.

  Since $U$ is surjective, there is $C>0$, such that $ \forall z \in \mathcal{Z},\| U^*(z) \|_\mathcal{H} \geq C \| z \|_\mathcal{Z} $. Therefore, $A = \inf \{ \| U^*(z) \|^2_\mathcal{H}:  \|z \|_\mathcal{Z} =1 \} $ is a positive constant. Therefore,
 $$ \langle Sz,z\rangle_\mathcal{Z} =\sum_{k \geq 0}  \langle V^k U U^* (V^*)^k(z),z\rangle_\mathcal{Z}=\sum_{k \geq 0}  \| U^* (V^*)^k(z) \|^2_\mathcal{H} \geq \| U^*(z) \|_\mathcal{H}^2 \geq A \| z \|_\mathcal{Z}^2.$$
 We also have
$$ \langle Sz,z\rangle_\mathcal{Z} =  \sum_{k \geq 0}  \| U^* (V^*)^k(z) \|^2_\mathcal{H} \leq \frac{\|U \|^2}{1-\|V\|^2} \|z\|_\mathcal{Z}^2=B,$$
and so  $(\{x_k= V^k U\}_{k \geq 0}, \mathfrak{b}_0)$ is a b-frame with the  b-frame bounds $A$ and $B$.

\emph{(2)} If $U U^*=I_{\mathcal{Z}}$ and $V$ is a normal operator, then
  \begin{eqnarray*}
S= \sum_{k \geq 0}   (VV^*)^{k}= (I_{\mathcal{Z}} -VV^*)^{-1},
\end{eqnarray*}
and so
\begin{eqnarray*}
A = \frac{1}{\|I_{\mathcal{Z}} -VV^* \| }  \mbox{ and } B=||( I_\mathcal{Z} -VV^* )^{-1}||.
\end{eqnarray*}
are  b-frame bounds for $(\{x_k= V^k U\}_{k \geq 0}, \mathfrak{b}_0)$.
\end{proof}

\begin{example}\label{GramianE}
We consider the restriction operator $U: h \in L^2(\mathbb{R}) \to Uh \in L^2([0,1])$, i.e.,
$ Uh(t) = h(t)$ for almost all $ t \in [0,1]$. We have then $\|U\|=1$ and $UU^*=I_{L^2([0,1])}$. Indeed, $U^*$ is the extension operator, i.e.,  if $z \in L^2([0,1])$, then $supp(U^*z) \subset [0,1]$ and, for almost all $t \in [0,1], U^*z(t)= z(t)$, and so $\|U^*z\|_{L^2(\mathbb{R})} =\|z \|_{L^2([0,1])}$.

Let $\varphi \in L^\infty([0,1])$ such that $\| \varphi \|_{L^\infty([0,1])} < 1$. The multiplication operator $ M_{\varphi}z= \varphi z,   z \in L^2([0,1]) $ is normal and $ \|M_{\varphi}\|= \| \varphi \|_{L^\infty([0,1])} < 1$. We also have $ M_{\varphi}^{k}z= \varphi^k z,  \forall z \in L^2([0,1]) $.

Using the second Statement of Proposition \ref{Gramian}, $(\{M_{\varphi^k}U\}_{k \geq 0}, \mathfrak{b}_0)$ is a Gramian b-frame with the frame bounds
\begin{eqnarray*}
A = \frac{1}{\|1 -|\varphi|^2 \|_{L^\infty([0,1])} }  \mbox{ and } B=\left|\left|\frac{1}{1 -|\varphi|^2}\right|\right|_{L^\infty([0,1])}.
\end{eqnarray*}
\end{example}

%%%%%%%%%%%%%%%%%%%%%%%%%%%%%%%%%%%%%%%%
\section{B-Riesz bases }
\label{sec:3}

\begin{defn} Let $\{x_k\}_{k \in \mathcal{K}} \subset \mathcal{B}$.
\begin{enumerate}
  \item $\{x_k\}_{k \in \mathcal{K}}$ is a b-basis for $\mathcal{Z}$ if each $z \in \mathcal{Z}$ has a unique representation in the form
  $$z= \sum_{k \in \mathcal{K}}  \mathfrak{b}(h_k,x_k). $$
  \item $\{x_k\}_{k \in \mathcal{K}}$  is b-orthonormal for $\mathcal{Z}$ if
  \begin{eqnarray*}
  \langle \mathfrak{b}(h,x_k) / x_j \rangle=\delta_{kj}h, \quad \forall j,k \in \mathcal{K} \mbox{ and } \forall h \in \mathcal{H}.
  \end{eqnarray*}
  \item $\{x_k\}_{k \in \mathcal{K}}$  is a b-orthonormal basis for $\mathcal{Z}$ if it is a b-basis and b-orthonormal.
  \item $\{x_k\}_{k \in \mathcal{K}}$ is a b-Riesz basis for $\mathcal{Z}$ if it is a b-frame and a b-basis for $\mathcal{Z}$.
\end{enumerate}
\end{defn}

The following lemma is proven in \cite{Ism}.
\begin{lem}\label{RieszL1} The system $\{x_k\}_{k \in \mathcal{K}}$ is a b-frame for $\mathcal{Z}$ if and only if there is a b-orthonormal basis $\{E_k\}_{k \in \mathcal{K}}$ for  $\mathcal{Z}$ and a bounded surjective operator $ U: \mathcal{Z} \rightarrow \mathcal{Z}$ such that $$Ub(h,E_k)=b(h,x_k), \forall h \in \mathcal{H}.$$
\end{lem}

\begin{lem}\label{RieszL2}Let $\{e_j\}_{j \in \mathcal{K}}$ be an orthonormal basis  for $\mathcal{H}$.  A system $\{x_k\}_{k \in \mathcal{K}}$ is  a b-orthonormal basis for  $\mathcal{Z}$ if and only if $\{ \mathfrak{b}(e_j,x_k)\}_{j,k \in \mathcal{K}}$ is an orthonormal basis for  $\mathcal{Z}$.
\end{lem}

\begin{proof} Suppose $\{x_k\}_{k \in \mathcal{K}}$ is  a b-orthonormal basis for  $\mathcal{Z}$. We have
\begin{eqnarray*}
\langle \mathfrak{b}(e_j,x_k),\mathfrak{b}(e_{j'},x_{k'}) \rangle_\mathcal{Z} &=& \langle e_j, \langle \mathfrak{b}(e_{j'},x_{k'}) /x_k  \rangle \rangle_\mathcal{H}\\ &=& \langle e_j,   \delta_{kk'}e_{j'} \rangle_\mathcal{H}= \delta_{kk'}\delta_{jj'},
\end{eqnarray*}
and by Proposition \ref{bdual}, $\{ \mathfrak{b}(e_j,x_k)\}_{j,k \in \mathcal{K}}$ is a frame for  $\mathcal{Z}$. We conclude that $\{ \mathfrak{b}(e_j,x_k)\}_{j,k \in \mathcal{K}}$ is an orthonormal basis for  $\mathcal{Z}$.

Now suppose $\{ \mathfrak{b}(e_j,x_k)\}_{j,k \in \mathcal{K}}$ is an orthonormal basis for  $\mathcal{Z}$. We have
\begin{eqnarray*}
\langle \mathfrak{b}(h,x_k)/x_j \rangle &=& \sum_{j' \in \mathcal{K}}\langle\langle \mathfrak{b}(h,x_k)/x_j) \rangle, e_{j'}  \rangle_\mathcal{H}e_{j'}\\
&=& \sum_{j' \in \mathcal{K}}\langle \mathfrak{b}(h,x_k), \mathfrak{b}(e_{j'},x_j) \rangle_\mathcal{Z} e_{j'}\\
&=& \sum_{j' \in \mathcal{K}}\sum_{k' \in \mathcal{K}} \langle  \mathfrak{b}(e_{k'},x_k), \mathfrak{b}(e_{j'},x_j) \rangle_\mathcal{Z} \langle  h, e_{k'} \rangle_\mathcal{H} e_{j'}\\
&=& \sum_{j' \in \mathcal{K}} \delta_{kj} \langle  h, e_{j'} \rangle_\mathcal{H} e_{j'}=\delta_{kj} h,\\
\end{eqnarray*}
and by Proposition \ref{bdual}, $(\{x_k\}_{k \in \mathcal{K}}$ is a b-frame for  $\mathcal{Z}$. We conclude that $(\{x_k\}_{k \in \mathcal{K}}$ is a b-orthonormal basis for  $\mathcal{Z}$.
\end{proof}

\begin{thm}\label{Riesz} Let $\{x_k\}_{k \in \mathcal{K}} \subset \mathcal{B}$ and let $\{e_j\}_{j \in \mathcal{K}}$ be an orthonormal basis  of $\mathcal{H}$.  The following statements are equivalent.
\begin{enumerate}
  \item $\{x_k\}_{k \in \mathcal{K}}$ is a b-Riesz basis for $\mathcal{Z}$.
  \item There is a b-orthonormal basis $\{E_k\}_{k \in \mathcal{K}}$ for  $\mathcal{Z}$ and a bounded bijective operator $ U: \mathcal{Z} \rightarrow \mathcal{Z}$ such that
      \begin{eqnarray}\label{Riesz1}
      U\mathfrak{b}(h,E_k)=\mathfrak{b}(h,x_k), \forall h \in \mathcal{H} \mbox{ and } \forall k \in \mathcal{K}
      \end{eqnarray}
  \item $\{b(e_j,x_k)\}_{j,k \in \mathcal{K}}$ is a Riesz basis for $\mathcal{Z}$.
\end{enumerate}
\end{thm}

\begin{proof}
(i) Suppose $\{x_k\}_{k \in \mathcal{K}}$  is a b-frame for $\mathcal{Z}$. To prove the equivalence of Statements \emph{(1)} and \emph{(2)}, we just need to prove that the operator $U$ in Lemma \ref{RieszL1} is one-to-one if and only if $\{x_k\}_{k \in \mathcal{K}}$  is a b-basis for $\mathcal{Z}$.

Let $z \in \mathcal{Z}$. Since $\{E_k\}_{k \in \mathcal{K}}$ is a b-orthonormal basis for $\mathcal{Z}$, $z$ has a unique representation in the form
\begin{eqnarray}\label{Riesz2}
z= \sum_{k \in \mathcal{K}}b(h_k,E_k).
\end{eqnarray}
Therefore,
\begin{eqnarray}\label{Riesz3}
Uz= \sum_{k \in \mathcal{K}} U b(h_k,E_k)= \sum_{k \in \mathcal{K}}  b(h_k,x_k).
\end{eqnarray}

Suppose $\{x_k\}_{k \in \mathcal{K}}$  is a b-basis for $\mathcal{Z}$ and $Uz=0$. In this case, \ref{Riesz3} implies   $h_k=0, \quad \forall k>0$, and so  $z=0$ by \ref{Riesz2}. Therefore, $U$ is one-to-one.

Conversely, suppose $\{x_k\}_{k \in \mathcal{K}}$  is not a b-basis for $\mathcal{Z}$. Therefore, we have
\begin{eqnarray}\label{Riesz4}
0= \sum_{k \in \mathcal{K}} b(h_k,x_k),
\end{eqnarray}
in which $h_k \neq 0$ for some indices $k \in \mathcal{K}$. Using the fact that $\{E_k\}_{k \in \mathcal{K}}$ is a b-orthonormal basis, $z$ defined by \ref{Riesz2} is not equal to zero, while $Uz=0$ by \ref{Riesz4}. Hence, $U$ is not one-to-one.

(ii) Suppose $\{x_k\}_{k \in \mathcal{K}}$ is a b-Riesz basis for $\mathcal{Z}$. Using the equivalence of Statements \emph{(1)} and \emph{(2)}, there is a b-orthonormal basis $\{E_k\}_{k \in \mathcal{K}}$ for  $\mathcal{Z}$ and a bounded bijective operator $ U: \mathcal{Z} \rightarrow \mathcal{Z}$ such that \ref{Riesz1} holds. Therefore,
\begin{eqnarray*}
      U\mathfrak{b}(e_j,E_k)=\mathfrak{b}(e_j,x_k), \quad \forall j,k \in \mathcal{K},
\end{eqnarray*}
which implies that $\{b(e_j,x_k)\}_{j,k \in \mathcal{K}}$ is a Riesz basis for $\mathcal{Z}$ since, by Lemma \ref{RieszL2}, $\{b(e_j,E_k)\}_{j,k \in \mathcal{K}}$ is an orthonormal basis for $\mathcal{Z}$.

Conversely, suppose $\{b(e_j,x_k)\}_{j,k \in \mathcal{K}}$ is a Riesz basis for $\mathcal{Z}$. By Proposition \ref{bdual}, $\{x_k\}_{k>0}$ is a b-frame. Therefore, by Lemma \ref{RieszL1}, there is a b-orthonormal basis $\{E_k\}_{k \in \mathcal{K}}$ for  $\mathcal{Z}$ and a bounded surjective operator $ U: \mathcal{Z} \rightarrow \mathcal{Z}$ such that \ref{Riesz1} holds.  Therefore,
\begin{eqnarray*}
      U\mathfrak{b}(e_j,E_k)=\mathfrak{b}(e_j,x_k), \forall j,k \in \mathcal{K}.
\end{eqnarray*}
The last equation and the fact that $\{b(e_j,x_k)\}_{j,k \in \mathcal{K}}$ is a Riesz basis imply that $U$ must be invertible. Using the equivalence of Statement \emph{(1)} and \emph{(2)},
$\{x_k\}_{k \in \mathcal{K}}$ is a b-Riesz basis. Hence, we finished the proof of the equivalence of Statements \emph{(2)} and \emph{(3)}.

\end{proof}

\begin{example}\label{RieszE} For $s \in \mathbb{R}$ and $f$ a measurable function, we define $T_sf(t)=f(t-s)$.
\begin{enumerate}
  \item Let $\alpha >0$, $\mathcal{H}=L^2([0,\alpha]), \mathcal{Z}= L^2(\mathbb{R}), \mathcal{B}= \mathcal{L} (\mathcal{H},\mathcal{Z})$, and $\mathfrak{b}(h, x)=x(h)$. Let  $\varphi \in L^\infty(\mathbb{R})$  such that $ supp (\varphi) \subset [0,\alpha]$ and
$$\forall t \in [0,\alpha],  \quad 0<a \leq | \varphi(t) | \leq b$$
For $  k \in \mathbb{Z}$ and $h \in \mathcal{H}$, we define
\[  x_k(h)(t)= \left\lbrace
  \begin{array}{c l}
    \varphi(t-k\alpha) h(t-k\alpha), & t \in [k\alpha , (k+1)\alpha] \\
   0, &  t \notin [k\alpha , (k+1)\alpha]
  \end{array}
\right. \]
We then have $\langle z/x_k\rangle = \overline{ \varphi} T_{-k\alpha} z$. Therefore,
\begin{eqnarray*}
\sum_{k \in \mathbb{Z}} \| \langle z/x_k\rangle \|_\mathcal{H}^2= \sum_{k \in \mathbb{Z}} \int |\overline{ \varphi}(t) z(t+k\alpha)|^2  dt= \sum_{k \in \mathbb{Z}} \int_{k\alpha} ^ {(k+1)\alpha }|\overline{ \varphi}(t-k\alpha) z(t)|^2 dt,
\end{eqnarray*}
and so
\begin{eqnarray*}
a^2 \| z \|_\mathcal{Z}^2= a^2 \sum_{k \in \mathbb{Z}} \int_{k\alpha } ^ {(k+1)\alpha }| z(t)|^2 dt \leq \sum_{k \in \mathbb{Z}} \| \langle z/x_k\rangle \|_\mathcal{H}^2 \leq b^2 \| z \|_\mathcal{Z}^2= b^2 \sum_{k \in \mathbb{Z}} \int_{k\alpha } ^ {(k+1)\alpha }| z(t)|^2 dt.
\end{eqnarray*}
We conclude that $\{ x_k \}_{ k \in \mathbb{Z}}$ is a b-frame with the bounds $A=a^2$ and $B=b^2$.

Now, let $$ \Phi(t)= \sum_{k \in \mathbb{Z}} | \varphi(t-k\alpha)|^2.$$
Hence, $\Phi$ is periodic with the period $ \alpha $, bounded with $A$ and $B$, and
\begin{eqnarray*}
\forall t \in [0, \alpha], \quad \Phi(t)= | \varphi (t)|^2.
\end{eqnarray*}

The operator
\begin{eqnarray*}
U : L^2(\mathbb{R})  & \to &  L^2(\mathbb{R})\\
z & \to & U(z)= \frac{z}{\Phi}
\end{eqnarray*}
is positive and invertible.  For $j, k \in \mathbb{Z}$, we have
\begin{eqnarray*}
 \langle U \mathfrak{b}(h,x_k)/x_j \rangle (t)&=& \overline{ \varphi} T_{-j\alpha}( \frac{T_{k\alpha}(\varphi h)}{\Phi})(t)\\
 &=& \overline{ \varphi}(t)   \frac{\varphi(t-(k-j)\alpha) h(t-(k-j)\alpha)}{\Phi(t+j\alpha)}\\
  &=&    \frac{\overline{ \varphi}(t)\varphi(t-(k-j)\alpha) h(t-(k-j)\alpha)}{\Phi(t)}\\
   &=& \delta_{jk} h(t).
\end{eqnarray*}
Thus, $\{ x_k \}_{ k \in \mathbb{Z}}$ is a b-Riesz basis.
\item The classic orthonormal basis for $L^2([0,\alpha])$ is $\{  \frac{1} {\sqrt{\alpha}}e^{\frac{2 \pi i nt}{\alpha}}  \}_{n \in \mathbb{Z}} $.
Using Lemma \ref{RieszL2},
$$\{  \frac{\varphi(t-k\alpha)} {\sqrt{\alpha}}e^{\frac{2 \pi i n(t-k\alpha)}{\alpha}}  \}_{k,n \in \mathbb{Z}} $$ is a Riesz basis for $L^2(\mathbb{R})$.
\item If $  \varphi(t) =1, \quad \forall t \in [0,\alpha]$, then $\{ x_k \}_{ k \in \mathbb{Z}}$ is a b-orthonormal basis.
\end{enumerate}
\end{example}

%$\varphi$ measurable, $ supp (\varphi) \subset [0,\alpha]$, and $ 0<a \leq \varphi \leq b < \infty$.  Then $\{  \frac{\varphi(t-k\alpha)} {\sqrt{\alpha}}e^{\frac{2 \pi i n(t-k\alpha)}{\alpha}}  \}_{k,n \in \mathbb{Z}}$ is a Riesz basis for $L^2(\mathbb{R})$.

%%%%%%%%%%%%%%%%%%%%%%%%%%%%%

\section{The tensor Product of b-frames  }
\label{sec:4}
In this section, for each $ i \in \{1,2\}$,  $\mathcal{H}_i$ and $\mathcal{Z}_i$ are two Hilbert spaces and $\mathcal{B}_i$ is a Banach space. The tensor product of frames and the tensor product of Riesz bases are characterized  in the following theorem by A. Bourouihiya.
\begin{thm}[\cite{Bou}] \label{bou} For each $ i \in \{1,2\}$, let $\{ z^i_k \}_{k \in \mathcal{K}} \subset \mathcal{Z}_i$.

 $\{ z^1_j \otimes z^2_k\}_{j,k \in \mathcal{K}}$ is a frame (or Riesz basis) for $\mathcal{Z}_1 \otimes \mathcal{Z}_2$ if and only if $\{ z^i_k\}_{k \in \mathcal{K}}$ is a frame (or Riesz basis) for $\mathcal{Z}_i $ for each $ i \in \{1,2\}$.

In addition the frame bounds of a frame $\{ z^1_j \otimes z^2_k\}_{j,k \in \mathcal{K}}$ are $A_1A_2$ and $B_1B_2$ if, for each $ i \in \{1,2\}$, the frame bounds of the frame $\{ z^i_k \}_{k \in \mathcal{K}} $ are $A_i$ and $B_i$.
\end{thm}

For $(h^1,h^2) \in  \mathcal{H}_1 \times \mathcal{H}_2, (z^1,z^2) \in  \mathcal{Z}_1 \times \mathcal{Z}_2,$ and $(x^1,x^2) \in  \mathcal{B}_1 \times \mathcal{B}_2,$ we define
\begin{eqnarray}\label{tensorbilinear}
\mathfrak{b}_1 \otimes \mathfrak{b}_2 ( h^1\otimes h^2, x^1\otimes x^2)= \mathfrak{b}_1(h^1,x^1) \otimes \mathfrak{b}_2 ( h^2,x^2).
\end{eqnarray}
This definition can algebraically be extended to define a bilinear mapping on vector spaces. Then, using the projective tensor product for the Banach spaces,
definition \ref{tensorbilinear} can  further be extended by completion  to define a unique bounded bilinear mapping $ \mathfrak{b}_1 \otimes \mathfrak{b}_2:  \mathcal{H}_1 \otimes \mathcal{H}_2 \times  \mathcal{B}_1 \otimes \mathcal{B}_2 \rightarrow \mathcal{Z}_1 \otimes \mathcal{Z}_2$ \cite{Rya}.

\begin{thm}\label{tensorbframe} For each $ i \in \{1,2\}$, let $\{ x^i_k \}_{k \in \mathcal{K}} \subset \mathcal{B}_i$.
\begin{enumerate}
\item $\{ x^1_j \otimes x^2_k\}_{j,k \in \mathcal{K}}$ is a b-frame (or b-Riesz basis) for $\mathcal{Z}_1 \otimes \mathcal{Z}_2$ if and only if $\{ x^i_k\}_{k \in \mathcal{K}}$ is a b-frame (or b-Riesz basis) for $\mathcal{Z}_i $ for each $ i \in \{1,2\}$.

\item if, for each $ i \in \{1,2\}$, the b-frame bounds of the frame $\{ x^i_k \}_{k \in \mathcal{K}} $ are $A_i$ and $B_i$, then  the b-frame bounds of the b-frame $\{ x^1_j \otimes x^2_k\}_{j,k \in \mathcal{K}}$ are $A_1A_2$ and $B_1B_2$.
\end{enumerate}
\end{thm}
\begin{proof} For each $ i \in \{1,2\}$, let $\{ e^i_j \}_{j \in \mathcal{K}}$ be an orthonormal basis in $\mathcal{H}_i$.  By \ref{tensorbilinear}, We have
\begin{eqnarray}\label{tensorbframe1}
 \mathfrak{b}_1 \otimes \mathfrak{b}_2 ( e^1_{j_1}\otimes e^2_{j_2}, x^1_{k_1}\otimes x^2_{k_2})= \mathfrak{b}_1(e^1_{j_1},x^1_{k_1}) \otimes \mathfrak{b}_2 (e^2_{j_2},x^2_{k_2}).
\end{eqnarray}

It is known that $\{ e^1_{j_1}\otimes e^2_{j_2}\}_{j_1,j_2 \in \mathcal{K}}$ is an orthonormal basis for $\mathcal{H}_1 \otimes \mathcal{H}_2$.  We consider the following statements.
\begin{enumerate}[(i)]
  \item  $\{ x^1_j \otimes x^2_k\}_{j,k \in \mathcal{K}}$ is a b-frame (or b-Riesz basis) for $\mathcal{Z}_1 \otimes \mathcal{Z}_2$.
  \item $ \{ \mathfrak{b}_1 \otimes \mathfrak{b}_2 ( e^1_{j_1}\otimes e^2_{j_2}, x^1_{k_1}\otimes x^2_{k_2}) \}_{j_1,j_2, k_1,k_2  \in \mathcal{K}}$ is a frame  (or Riesz basis)  for $\mathcal{Z}_1 \otimes \mathcal{Z}_2$.
  \item $ \{ \mathfrak{b}_1(e^1_{j_1},x^1_{k_1}) \otimes \mathfrak{b}_2 (e^2_{j_2},x^2_{k_2}) \}_{j_1,j_2, k_1,k_2  \in \mathcal{K}}$ is a frame (or Riesz basis)  for $\mathcal{Z}_1 \otimes \mathcal{Z}_2$.
  \item $ \{ \mathfrak{b}_1(e^1_{j_1},x^1_{k_1}) \}_{j_1, k_1  \in \mathcal{K}}$ is a  frame  (or Riesz basis) for $\mathcal{Z}_1 $ and $ \{  \mathfrak{b}_2 (e^2_{j_2},x^2_{k_2}) \}_{j_2,k_2  \in \mathcal{K}}$ is a frame (or Riesz basis)  for $\mathcal{Z}_2$.
  \item $\{ x^1_k\}_{k \in \mathcal{K}}$ is a b-frame  (or b-Riesz basis)  for $\mathcal{Z}_1 $ and $\{ x^2_k\}_{k \in \mathcal{K}}$ is a b-frame  (or b-Riesz basis)  for $\mathcal{Z}_2$.
\end{enumerate}

By Proposition \ref{frame}, (i) is equivalent to (ii). Using \ref{tensorbframe1}, (ii) is equivalent to (iii). By Theorem \ref{bou}, (iii) is equivalent to (iv). Finally, by Proposition \ref{frame}, (iv) is equivalent to (v).  Hence, Statement \emph{(1)} is proven.

To prove Statement \emph{(2)}, we use Theorem \ref{bou}.
\end{proof}

%%%%%%%%%%%%%%%%%%%%%%%%%%%%%%%%%%%%%%%
\section{B-frames and G-frames }
\label{sec:5}
\begin{thm}\label{bgframe}
\begin{enumerate}
\item  Let $\{x_k\}_{k\in \mathcal{K}} \subset \mathcal{B}$ and let $\mathfrak{b}: \mathcal{H} \times \mathcal{B} \to \mathcal{Z}$ be a continuous bilinear mapping.

    $(\{x_k\}_{k\in \mathcal{K}},\mathfrak{b})$ is a b-frame (b-Riesz basis) if and only if $\{\Lambda_k: z  \rightarrow \langle z / x_k \rangle \}_{k\in \mathcal{K}}$ is a g-frame (g-Riesz basis) with respect to the sequence of Hilbert spaces $ \{\mathcal{Z}_k = \overline{range(\Lambda_k)}\}_{k\in \mathcal{K}} $.
\item  Let $ \{\mathcal{Z}_k \}_{k\in \mathcal{K}} $ be a sequence of Hilbert spaces, let $\{\Lambda_k  \in \mathcal{L}(\mathcal{Z},\mathcal{Z}_k)\}_{k\in \mathcal{K}} $, and  Let\\ $ \mathcal{H}=\bigoplus_{k>0} \mathcal{Z}_k.$

$\{\Lambda_k \}_{k\in \mathcal{K}} $  is a g-frame (g-Riesz basis) with respect to  $ \{\mathcal{Z}_k \}_{k\in \mathcal{K}} $ if and only if\\  $(\{x_k:\mathcal{H} \to \mathcal{Z} \}_{k\in \mathcal{K}},\mathfrak{b}_0)$ is a b-frame (b-Riesz basis), where $\forall h \in \mathcal{Z}_k, x_k(h)=\Lambda_k^*(h)$ and $\ker (x_k)= \mathcal{Z}_k^\perp$.
\end{enumerate}
For each of the Statements (1) and (2),  the b-frame and g-frame share the same bounds and the b-frame operator coincide with the  g-frame operator.
 \end{thm}

 \begin{proof} \emph{(1)}  For each $k\in \mathcal{K}$, $\mathcal{Z}_k$ is a sub-Hilbert space of $\mathcal{H}$. For $(h,z) \in \mathcal{Z}_k \times \mathcal{Z}$, we have
 \begin{eqnarray*}
 \langle h,\Lambda_k(z) \rangle_{\mathcal{Z}_k}= \langle  h , \langle z / x_k \rangle \rangle_{\mathcal{Z}_k} =\langle \mathfrak{b}(h,x_k), z \rangle_{\mathcal{Z}}.
\end{eqnarray*}
Therefore, $\Lambda_k^*(h)=\mathfrak{b}(h,x_k)$, which is a bounded operator. Hence, \ref{bframe}  holds for $(\{x_k\}_{k\in \mathcal{K}},\mathfrak{b})$ if and only if \ref{gframe} holds for $\{\Lambda_k \}_{k\in \mathcal{K}}$, and so $(\{x_k\}_{k\in \mathcal{K}},\mathfrak{b})$ is a b-frame  if and only if $\{\Lambda_k \}_{k\in \mathcal{K}}$ is a g-frame  with respect to $ \{\mathcal{Z}_k \}_{k\in \mathcal{K}} $. In addition, the b-frame bounds for $(\{x_k\}_{k\in \mathcal{K}},\mathfrak{b})$ coincide with the  g-frame bounds for $\{\Lambda_k \}_{k\in \mathcal{K}}$.  We also have
\begin{eqnarray*}
S(z)=\sum_{k>0} \mathfrak{b}(\langle z / x_k \rangle,x_k )=\sum_{k\in \mathcal{K}} \Lambda_k^*\Lambda_k(z),
\end{eqnarray*}
and so $S$ is the b-frame operator for $(\{x_k\}_{k\in \mathcal{K}},\mathfrak{b})$ and the g-frame operator for $\{\Lambda_k \}_{k\in \mathcal{K}}$.

Now Suppose that $(\{x_k\}_{k\in \mathcal{K}},\mathfrak{b})$ is a b-frame, and so $\{\Lambda_k \}_{k\in \mathcal{K}}$ is a g-frame, as it is now proven. Using \ref{operator}, for each $z \in \mathcal{Z}$, we have
\begin{eqnarray*}
z=\sum_{k\in \mathcal{K}} \mathfrak{b}(\langle S^{-1}z / x_k \rangle,x_k )=\sum_{k\in \mathcal{K}} \Lambda_k^*\Lambda_kS^{-1}(z).
\end{eqnarray*}
Therefore,
\begin{eqnarray*}
z=\sum_{k\in \mathcal{K}} \mathfrak{b}(h_k,x_k )=\sum_{k\in \mathcal{K}} \Lambda_k^*(h_k),
\end{eqnarray*}
where $\forall k\in \mathcal{K}, h_k =\langle S^{-1}z / x_k \rangle=\Lambda_kS^{-1}(z) \in \mathcal{Z}_k$. Hence, $\sum_{k\in \mathcal{K}} \mathfrak{b}(h_k,x_k )$ is the unique expansion of $z$ with respect to $(\{x_k\}_{k\in \mathcal{K}},\mathfrak{b})$ if and only if $ \sum_{k\in \mathcal{K}} \Lambda_k^*(h_k)$ is the unique expansion of $z$ with respect to $\{\Lambda_k \}_{k\in \mathcal{K}}$. We conclude that $(\{x_k\}_{k\in \mathcal{K}},\mathfrak{b})$ is a b-Riesz basis if and only if $\{\Lambda_k \}_{k\in \mathcal{K}}$ is a g-Riesz basis.

\emph{(2)}  We first precise that  $ \mathcal{H}=\bigoplus_{k\in \mathcal{K}} \mathcal{Z}_k$ is a direct sum of Hilbert spaces, i.e.,
$$ \mathcal{H}=\left\{ \sum_{k\in \mathcal{K}} h_k: \forall k\in \mathcal{K}, h_k \in \mathcal{Z}_k \mbox{ and }   \sum_{k\in \mathcal{K}} \|h_k\|^2 < \infty \right\},$$
where $h_k \sim (0, \dots, 0,h_k,0, \dots)$ ($h_k$ at the $k^{th}$ position), endowed with the inner product
$$\left\langle \sum_{k\in \mathcal{K}} h_k,\sum_{k>0} h'_k \right\rangle_{\mathcal{H}} = \sum_{k\in \mathcal{K}} \langle h_k,h'_k \rangle_{\mathcal{Z}_k}.$$
 It is known that $\mathcal{H}$ is a Hilbert space and $\mathcal{Z}_j$ is orthogonal to $\mathcal{Z}_k$ if $j \neq k$ \cite{Con}.

Let $h=\sum_{j\in \mathcal{K}} h_j \in \mathcal{H}$, where $\forall j \in \mathcal{K}, h_j \in \mathcal{Z}_j$.  Since $x_k(l)=\Lambda_k^*(l)$ if $l \in \mathcal{Z}_k$,  and $x_k(l)=0$ if $l \in \mathcal{Z}_j$ and $j \neq k$, we have $x_k(h)=\Lambda_k^*(h_k)$. For $z \in \mathcal{Z}$, we then have
\begin{eqnarray*}
 \langle h , x_k^*(z) \rangle_{\mathcal{H}}=  \langle x_k(h) , z \rangle_{\mathcal{Z}}= \langle \Lambda_k^*(h_k) , z \rangle_{\mathcal{Z}}= \langle h_k.  \Lambda_k(z) \rangle_{\mathcal{Z}_k} =\langle h , \Lambda_k( z )\rangle_{\mathcal{H}}.
\end{eqnarray*}
The last equality holds because $\mathcal{Z}_j$ is orthogonal to $\mathcal{Z}_k$ if $j \neq k$. Therefore, $x_k^*=\Lambda_k$, and so $\langle z / x_k\rangle_{\mathfrak{b}_0}= x_k^*(z)=\Lambda_k(z)$. Following, the same steps in the proof of Statement \emph{(1)}, we can finish the proof of Statement \emph{(2)}.

For each of the Statements \emph{(1)} and \emph{(2)}, it is clear that the b-frame and g-frame have the same bounds and the b-frame operator coincide with the  g-frame operator.
\end{proof}

%%%%%%%%%%%%%%%%%%%%%%%%%%%%%%%%%%%%%%%
\section{Frames induced by  b-frames }
\label{sec:6}
Generalizing frames allows researchers to prove properties of standard frames \cite{Kho1} and use the generalized framework to build examples of frames \cite{Chr, Sun}.
Statement \emph{(1)} of Proposition 2.1 allows one to construct frames for $\mathcal{Z}$ induced by  b-frames for $\mathcal{Z}$ and orthonormal bases for $\mathcal{H}$. We used this fact, in Example \ref{RieszE}, to construct a Riesz basis for $L^2(\mathbb{R})$ using a b-Riesz basis for $L^2(\mathbb{R})$ and an orthonormal basis for $L^2([0,\alpha])$. However, orthonormal bases are known to be generally impractical for applications. In this section, we prove that frames for $\mathcal{Z}$ can be induced  by b-frames for $\mathcal{Z}$ and a frames for $\mathcal{H}$.

\begin{thm}\label{induced} Let $\{f_j\}_{j \in \mathcal{K}}$  be a frame for  $\mathcal{H}$, with the frame bounds $A$ and $B$, and let $\{\tilde{f_j}\}_{j \in \mathcal{K}}$ be its canonical dual frame. Let $\{x_k\}_{k \in \mathcal{K}} \subset \mathcal{B}$.
\begin{enumerate}
\item $(\{x_k\}_{k \in \mathcal{K}},\mathfrak{b})$ is a b-frame for $\mathcal{Z}$ with the b-frame bounds $A_b$ and $B_b$, if and only if   $\{\mathfrak{b}(f_j,x_k)\}_{j,k \in \mathcal{K}}$ is a  frame for $\mathcal{Z}$ with the frame bounds $AA_b$ and $BB_b$.
\item Suppose that $\{f_j\}_{j \in \mathcal{K}}$ is a Riesz basis.

  $(\{x_k\}_{k \in \mathcal{K}},\mathfrak{b})$ is a  b-Riesz basis for $\mathcal{Z}$ if and only if $\{\mathfrak{b}(f_j,x_k)\}_{j,k \in \mathcal{K}}$ is a Riesz basis for $\mathcal{Z}$.

\end{enumerate}
 \end{thm}

\begin{proof}
\emph{(1)} Using the fact that $\{f_j\}_{j \in \mathcal{K}}$ is a frame for $\mathcal{H}$,  for each $k \in \mathcal{K} $ and each $z \in \mathcal{Z}$, we have
\begin{eqnarray*}
A \|\langle z/x_k \rangle \|_{\mathcal{H}}^2 \leq   \sum_{j\in \mathcal{K} } | \langle  \mathfrak{b}(f_j,x_k), z \rangle_{\mathcal{Z}}|^2=\sum_{j\in \mathcal{K} } | \langle  f_j, \langle z/x_k \rangle \rangle_{\mathcal{H}}|^2 \leq B \|\langle z/x_k \rangle \|_{\mathcal{H}}^2,
\end{eqnarray*}
and so
\begin{eqnarray}\label{induced1}
A \sum_{k\in \mathcal{K} } \|\langle z/x_k \rangle \|_{\mathcal{H}}^2 \leq   \sum_{j,k\in \mathcal{K} } | \langle  \mathfrak{b}(f_j,x_k), z \rangle_{\mathcal{Z}}|^2 \leq B\sum_{k\in \mathcal{K} } \|\langle z/x_k \rangle \|_{\mathcal{H}}^2.
\end{eqnarray}

\emph{(i)} If $(\{x_k\}_{k \in \mathcal{K}},\mathfrak{b})$ is a b-frame for $\mathcal{Z}$ with the b-frame bounds $A_b$ and $B_b$, then \ref{bframe}   and \ref{induced1}
yield
\begin{eqnarray*}
AA_b \|z \|_{\mathcal{Z}}^2 \leq   \sum_{j,k\in \mathcal{K} } | \langle  \mathfrak{b}(f_j,x_k), z \rangle_{\mathcal{Z}}|^2 \leq BB_b \|z \|_{\mathcal{Z}}^2.
\end{eqnarray*}
Hence, $\{\mathfrak{b}(f_j,x_k)\}_{j,k \in \mathcal{K}}$ is a frame for $\mathcal{Z}$ with the frame bounds $AA_b$ and $BB_b$.

\emph{(ii)} Conversely, suppose $\{\mathfrak{b}(f_j,x_k)\}_{j,k \in \mathcal{K}}$ is a  frame for $\mathcal{Z}$ with the frame bounds $A_0$ and $B_0$. Therefore, for each $z \in \mathcal{Z}$, we have
\begin{eqnarray*}
A_0 \|z \|_{\mathcal{Z}}^2 \leq   \sum_{j,k\in \mathcal{K} } | \langle  \mathfrak{b}(f_j,x_k), z \rangle_{\mathcal{Z}}|^2 \leq B_0 \|z \|_{\mathcal{Z}}^2.
\end{eqnarray*}
Using \ref{induced1} and the last inequalities, we obtain
\begin{eqnarray*}
\frac{A_0}{B} \|z \|_{\mathcal{Z}}^2 \leq   \sum_{k\in \mathcal{K} } \|\langle z/x_k \rangle \|_{H}^2  \leq \frac{B_0}{A} \|z \|_{\mathcal{Z}}^2.
\end{eqnarray*}

Thus, $(\{x_k\}_{k \in \mathcal{K}},\mathfrak{b})$ is a b-frame for $\mathcal{Z}$. Suppose this b-frame has the b-frame bounds $A_b$ and $B_b$. Using step  \emph{(i)}, the frame  $\{\mathfrak{b}(f_j,x_k)\}_{j,k \in \mathcal{K}}$ has the frame bounds $AA_b$ and $BB_b$.

\emph{(2)}  Suppose that $(\{x_k\}_{k \in \mathcal{K}},\mathfrak{b})$ is a b-Riesz basis.  Using Statement \emph{(1)}, $\{\mathfrak{b}(f_j,x_k)\}_{j,k \in \mathcal{K}}$ is then a frame.  Suppose that
$$ 0= \sum_{k \in \mathcal{K}} \sum_{j \in \mathcal{K}} \lambda_{j,k}\mathfrak{b}(f_j,x_k)=\sum_{k \in \mathcal{K}} \mathfrak{b}(\sum_{j \in \mathcal{K}} \lambda_{j,k}f_j,x_k).$$
The fact that $(\{x_k\}_{k \in \mathcal{K}},\mathfrak{b})$  is b-Riesz basis implies
$$ \forall k \in \mathcal{K}, \sum_{j \in \mathcal{K}} \lambda_{j,k}f_j=0,$$
and the fact that $\{f_j\}_{j \in \mathcal{K}}$ is a Riesz basis implies then that $ \forall j,k,  \lambda_{j,k}=0.$ Therefore, $\{\mathfrak{b}(f_j,x_k)\}_{j,k \in \mathcal{K}}$ is a Riesz basis.

Conversely, suppose that $\{\mathfrak{b}(f_j,x_k)\}_{j,k \in \mathcal{K}}$ is a Riesz basis.  Using  Statement \emph{(1)}, $(\{x_k\}_{k \in \mathcal{K}},\mathfrak{b})$ is then a b-frame.  Suppose that
$$ 0=\sum_{k \in \mathcal{K}} \mathfrak{b}(h_k,x_k)= \sum_{k \in \mathcal{K}} \mathfrak{b}(\sum_{j \in \mathcal{K}}\langle h_k,\tilde{f_j}\rangle_{\mathcal{H}} f_j,x_k)= \sum_{k \in \mathcal{K}} \sum_{j \in \mathcal{K}}\langle h_k,\tilde{f_j}\rangle_{\mathcal{H}}
\mathfrak{b}( f_j,x_k).$$
the fact that $\{\mathfrak{b}(f_j,x_k)\}_{j,k \in \mathcal{K}}$ is a Riesz basis implies $ \forall j,k, \langle h_k,\tilde{f_j}\rangle_{\mathcal{H}}=0$. Using the fact that $\{\tilde{f}_j\}_{j \in \mathcal{K}}$ is a Riesz basis implies that $ \forall k, h_k=0$. Therefore, $(\{x_k\}_{k \in \mathcal{K}},\mathfrak{b})$ is a b-Riesz basis. Hence, we finish the proof of Statement \emph{(2)}.

\end{proof}

\begin{rem}
One of the applications of g-frames is the construction of  frames for a Hilbert space induced by g-frames and frames in another Hilbert space. The main theorem allowing that in  \cite{Sun} uses a g-frame and a sequence of frames for which the sequences of frame upper and lower bounds are bounded away from zero and infinity. Meanwhile,  Theorem \ref{induced} requires no restriction on the b-frame to construct a frame for a Hilbert space induced by a g-frame and a frame in another Hilbert space. Using Theorems \ref{bgframe} and \ref{induced}, we can state a theorem on induced frames by a g-frame.
\end{rem}

\begin{example}
\begin{enumerate}
  \item  Let $\varphi(t)= \frac{e^{-t^2}}{2}$.  The b-frame $(\{ M_{\varphi^k}U\}_{k \geq 0}, \mathfrak{b}_0)$ defined in Example \ref{GramianE}, where $U: L^2(\mathbb{R}) \to L^2([0,1])$ is the restriction operator,   has the b-frame bounds
\begin{eqnarray*}
A_b = \frac{1}{\|1 -|\varphi|^2 \|_{L^\infty([0,1])} }= \frac{4e^2}{4e^2-1} \mbox{ and } B_b=||\frac{1}{1 -|\varphi|^2}||_{L^\infty([0,1])}=\frac{4}{3}.
\end{eqnarray*}

Let $A$ and $B$ be the frame bounds of the Gabor frame $\{ e^{2 \pi i n \beta t} e^{-\pi(t- m\alpha)^2}, m,n \in \mathbb{Z} \}$ for $L^2(\mathbb{R})$, where $\alpha \beta <1$. Values, in terms of $\alpha$ and  $\beta$, for $A$ and $B$ are published \cite{Gro}.

 Using Theorem \ref{induced},
  $\mathcal{F}=\{  e^{2 \pi i n \beta t} e^{-\pi(t- m\alpha)^2 -kt^2}, m,n \in \mathbb{Z} \mbox{ and } k \geq 0 \}$ is a frame for $L^2([0,1])$ with the frame bounds $AA_b$ and $BB_b$.

  A Direct calculation shows that $A$ and $\frac{4}{3}B$ are frame bounds for $\mathcal{F}$. Using Theorem \ref{induced}, the lower frame bound is $AA_b >A$, which is an improvement of the lower bound.

\end{enumerate}

\end{example}

\bigskip
\noindent
{\bf Acknowledgment.}
The authors would like to express their gratitude to Dr. Samir Kabbaj for his valuable discussions and suggestions, which helped improve the content of this paper.

%%%%%%%%%%%%%%%%%%%%%%%%%%%%%%%%%%%%%%%%%%%%%%%%%%%%%%%%%%


\begin{thebibliography}{99.}%
% Use the following syntax and markup for your references if
% the subject of your book is from the field
% "Mathematics, Physics, Statistics, Computer Science"
%
% Contribution
\bibitem{Bou} Bourouihiya, A., 2008. The tensor product of frames. Sampling theory in signal and Image processing, 7(1), pp.65-76.
\bibitem{Chr} Christensen, O., An Introduction to Frames and Riesz Bases, Birkh‰ user, Boston, 2003. MR1946982 (2003k: 42001).
\bibitem{Con} Conway, J.B., 2019. A course in functional analysis. Springer.
\bibitem{Gro} Gröchenig, K., 2001. Foundations of time-frequency analysis (Vol. 359). Boston: Birkhäuser.
\bibitem{Ism}Ismailov, M., Guliyeva, F. and Nasibov, Y., 2016. On a generalization of the Hilbert frame generated by the bilinear mapping. Journal of Function Spaces, 2016(1), p.9516839.
\bibitem{Kho}Khosravi, A. and Azandaryani, M.M., 2012. Fusion frames and g-frames in tensor product and direct sum of Hilbert spaces. Applicable Analysis and Discrete Mathematics, pp.287-303.
\bibitem{Kho1} Khosravi, A., Farmani, M. R. (2022). Frame and g-frame in Hilbert spaces. Mathematics and Computational Sciences, 3(1), 10-16.
\bibitem{Mez} Mezzat, C. and Kabbaj, S., 2024. $ K $-$ b $-Frames for Hilbert Spaces and the $ b $-Adjoint Operator. Sahand Communications in Mathematical Analysis, 21(4), pp.1-26.
\bibitem{Rya} Ryan, R.A. and a Ryan, R., 2002. Introduction to tensor products of Banach spaces (Vol. 73). London: Springer.
\bibitem{Son} Sondergaard, P.L., 2007. Gabor frames by sampling and periodization. Advances in Computational Mathematics, 27(4), pp.355-373.
\bibitem{Sun} Sun, W., 2006. G-frames and g-Riesz bases. Journal of Mathematical Analysis and Applications, 322(1), pp.437-452.
\bibitem{Zha} Zhang, J., Gao, F. and Hong, G., 2024. A Note on Injective g-Frames in Quaternionic Hilbert Spaces. Axioms, 13(12), p.851.
\end{thebibliography}
\end{document}